\documentclass[12pt]{amsart}
\usepackage{etex}

\usepackage{amsthm,amssymb,enumitem,mathtools, verbatim}

\usepackage{tikz}

\usetikzlibrary{arrows,shapes,automata,backgrounds,decorations,petri,positioning}

\tikzstyle{edge} = [fill,opacity=.5,fill opacity=.5,line cap=round, line join=round, line width=50pt]

\pgfdeclarelayer{background}
\pgfsetlayers{background,main}

\usepackage[margin=1in,letterpaper,portrait]{geometry}

\usepackage{amsthm}

\usepackage[latin1]{inputenc}
\usepackage{subfigure}
\usepackage{color}
\usepackage{amsmath}
\usepackage{amsthm}
\usepackage{amstext}
\usepackage{amssymb}
\usepackage{amsfonts}
\usepackage{graphicx}
\usepackage{young}
\usepackage{multicol}
\usepackage{mathrsfs}
\usepackage[all]{xy}
\usepackage{tikz}
\usepackage{mathtools}
\usepackage{tabularx}
\usepackage{array}
\usepackage{commath}
\usepackage{ytableau}
\usepackage{hyperref}

\theoremstyle{plain}
\theoremstyle{definition}
\newtheorem{theorem}{Theorem}[section]
\newtheorem{conjecture}[theorem]{Conjecture}
\newtheorem{remark}[theorem]{Remark}

\newtheorem{definition}[theorem]{Definition}

\newtheorem{example}[theorem]{Example}
\newtheorem{proposition}[theorem]{Proposition}
\newtheorem{corollary}[theorem]{Corollary}

\DeclareMathAlphabet{\mathpzc}{OT1}{pzc}{m}{it}

\renewcommand{\b}{\mathfrak{b}}
\renewcommand{\c}{\mathfrak{c}}

\newcommand{\ub}[1]{\underbracket[.5pt][1pt] {#1}} 

\newcommand{\omitt}[1]{}

\usepackage{amssymb,amsmath,tabularx,graphicx}
\usepackage{tikz}
\usetikzlibrary[calc,intersections,through,backgrounds,arrows,decorations.pathmorphing]

\usepackage[colorinlistoftodos]{todonotes}

\begin{document}

\title{Enumerations relating braid and commutation classes}

\date{}

\author[Fishel]{Susanna Fishel}
\author[Mili\'{c}evi\'{c}]{Elizabeth Mili\'{c}evi\'{c}}
\author[Patrias]{Rebecca Patrias}
\author[Tenner]{Bridget Eileen Tenner}

\thanks{SF was partially supported by grant 359602 from the Simons Foundation, as well as by funds from the Mathematical Sciences Research Institute. EM was partially supported by NSF grant DMS-1600982. RP was supported by NSERC, the Canada Research Chairs Program, and CRM-ISM. BT was partially supported by Simons Foundation Collaboration Grant for Mathematicians 277603 and a DePaul University Faculty Summer Research Grant. \textcopyright\ 2018. This manuscript version is made available under the CC-BY-NC-ND 4.0 license http://creativecommons.org/licenses/by-nc-nd/4.0/}

\address{Susanna Fishel, School of Mathematical and Statistical
  Sciences, Arizona State University, Tempe, AZ, USA}
\email{sfishel1@asu.edu}

\address{Elizabeth Mili\'{c}evi\'{c}, Department of Mathematics \& Statistics, Haverford College, Haverford, PA, USA}
\email{emilicevic@haverford.edu}

\address{Rebecca Patrias, 
LaCIM, Universit\'e du Qu\'ebec \`a Montr\'eal \\
Montr\'eal (Qu\'ebec), Canada}
\email{patriasr@lacim.ca}

\address{Bridget Eileen Tenner, 
Department of Mathematical Sciences, 
DePaul University, Chicago, IL, USA}
\email{bridget@math.depaul.edu}



\begin{abstract}
We obtain an upper and lower bound for the number of reduced words for a permutation in terms of the number of braid classes and the number of commutation classes of the permutation. We classify the permutations that achieve each of these bounds, and enumerate both cases.
\end{abstract}

\maketitle

\section{Introduction}

The symmetric group $S_n$ is the Coxeter group generated by the adjacent transpositions $\{s_1,\ldots,s_{n-1}\}$.  Every $w \in S_n$ can be expressed as a minimal-length product of these generators, and a given permutation may have more than one such representation. These representations are \emph{reduced expressions} for $w$, each of which can be encoded as a \emph{reduced word} for $w$. The set of reduced words $R(w)$ possesses a rich combinatorial structure that has been studied from many different perspectives. For example, Stanley showed that the number $|R(w)|$ of reduced words for $w$ can be calculated in terms of Young tableaux of particular shapes \cite{stanley1984}, but this is not an easy value to calculate outside of special cases. Another common technique explores various quotients of $R(w)$ under the relations governing the adjacent transpositions.  For example, the reduced words ${\bf u}$ and ${\bf v}$ for $w \in S_n$ are in the same \emph{commutation class} if we can obtain ${\bf u}$ from
${\bf v}$ by applying a sequence of commutation relations of the form $s_is_j=s_js_i$ for
$|i-j|>1$. The \emph{braid classes} of $w$ are defined similarly, in terms of the braid
relation $s_is_{i+1}s_i=s_{i+1}s_is_{i+1}$.   In this paper, we give upper and lower bounds for $|R(w)|$, in terms of the number of braid and commutation classes of the permutation $w$.

Quotients of $R(w)$ under these Coxeter relations have been studied before, but almost all of that work has focused on the commutation classes, with very little attention paid to the braid classes. In \cite{elnitsky1997rhombic}, Elnitsky presented
a bijection from the commutation classes of $w$ to rhombic tilings of a particular
polygon that depends on $w$, which Tenner utilized extensively in \cite{tenner-rdpp}.  Jonsson and
Welker \cite{jonsson-welker} considered certain cell complexes where
the commutation classes appear. Bergeron, Ceballos, and
Labb\'e \cite{bergeron2015fan} studied a similar scenario, but their work does consider quotienting out by more general Coxeter relations. Meng \cite{meng} studied both the number of commutation classes and their network relationships, and B{\'e}dard developed recursive formulas for the number of reduced words in each commutation class in \cite{bedard}, as well as more detailed statistics for the case of the longest permutation $w_0$.  Such focus on certain fixed elements has also been a productive approach to studying the set of reduced words. For example, the higher Bruhat order $B(n,2)$, studied by Manin and Schechtman \cite{manin-schechtman} and Ziegler \cite{ziegler} is a partial order on the commutation classes of $w_0$.

Another important family of permutations for which the study of the commutation classes has been especially fruitful are the fully commutative elements, which are those having a single commutation class. In their study of Schubert polynomials, Billey, Jockusch, and Stanley
\cite{billey-jockusch-stanley} showed that fully commutative
permutations are $321$-avoiding.  Stembridge investigated and enumerated the fully commutative elements in all Coxeter groups
in \cite{stembridge96,stembridge97,stembridge98}, and he provided further
examples of their appearance in algebra.  More recently, Green and
Losonczy \cite{green-losonczy2002} and Hanusa and Jones
\cite{hanusa-jones} have worked on adaptations of
full commutativity to root systems and affine permutations, respectively. 

The equivalence classes of $R(w)$ under the braid relations, on the other hand, have been noticeably less well studied, both in terms of how they partition
the collection of reduced words, and as objects of structure
themselves.  Zollinger \cite{zollinger1994equivalence} used an encoding of the reduced words to provide formulas for the size of the braid classes, and the previously cited work in \cite{bergeron2015fan} showed that the graph on braid classes, with edges indicating when an element of one class can be transformed into an element of the other class by a commutation move, is bipartite.

As with commutation classes, focusing on a particular element in $S_n$, such as the longest permutation $w_0$, has been a fruitful approach to understanding some of the relevance and influence of braid moves in reduced words.  Reiner and Roichman \cite{reiner-roichman} defined a graph whose vertices are the elements of $R(w_0)$ with edges indicating braid relations, and they calculate the diameter of this graph.  Building on Reiner's proof  \cite{reiner05} that the expected number of braid moves in a random element of $R(w_0)$ equals one, Schilling, Thi\'ery, White, and Williams \cite{schilling-etal} extend this probabilistic result to the case of one particular commutation class of reduced words for the longest element.

Although braid classes and commutation classes are highly related, recognition of this fact has thus far been under-utilized in the literature. In this paper, we seek to remedy that by studying the relationship between the number of reduced words of a permutation and the numbers of commutation classes and braid classes that it has.  Our driving principle is to utilize the fact that commutation
classes and braid classes are partitions of the same set. This allows us to leverage one against the other when studying them in tandem. Theorem \ref{thm:bounding |R|} illustrates this approach by providing upper and lower bounds on $|R(w)|$ in terms of the number of braid and commutation classes.  More precisely,
$$|B(w)| + |C(w)| - 1 \le |R(w)| \le |B(w)| \cdot |C(w)|,$$
where we define this notation in the next section. We demonstrate that these bounds are sharp by characterizing in Proposition~\ref{prop:characterizing the upper bound} (and Corollary~\ref{cor:characterizing upper bound}) and Proposition~\ref{prop:characterizing lower bound achievers} those permutations that achieve the upper and lower bounds, respectively.  Those collections of permutations are enumerated in Corollaries~\ref{cor:upper bound enumeration} and~\ref{cor:lower bound enumeration}.  In Conjecture \ref{conj:lower bound}, we suggest an alternate approach to studying these relationships in terms of intervals in the weak order, rather than working directly with the Coxeter relations on reduced words.

We focus on reduced words for elements of the symmetric group,  and Section~\ref{sec:preliminaries} provides the required background on Coxeter groups, as well as the relevant terminology and notation used throughout the paper.  In addition, we recall facts about the graph with vertex set $R(w)$ and edges indicating a single commutation or braid move, and we discuss initial results about the quantity and structure of braid classes.
This sets the stage for Section~\ref{sec:bounds}, which contains the majority of our results, including those discussed above. The paper concludes in Section~\ref{sec:conjecture} with a proposed connection to aspects of intervals in the weak order.

\subsection*{Acknowledgements} This project began at the ``Algebraic Combinatorixx 2'' workshop at the Banff International Research Station in May 2017. We thank the organizers and staff for putting together such a wonderful conference, and thank Kyoungsuk Park for early discussions.  The computations in this paper were performed by using Maple$^{\text{TM}}$, a trademark of Waterloo Maple Inc.~~John Stembridge's \textsf{posets} and \textsf{coxeter} packages for Maple were essential.

\section{Braid and commutation classes in Coxeter groups}
\label{sec:preliminaries}

We begin by briefly discussing relevant background from the theory of Coxeter groups in Section \ref{S:Coxeter}. For further background on Coxeter groups, we refer the reader to \cite{bjorner-brenti}.   In Section \ref{S:ClassDefs}, we define the braid and commutation classes in the symmetric group, and establish terminology related to braid and commutation moves. In Section \ref{sec:initialobservations}, we review a graphical representation for the braid and commutation classes and recall some of its properties.  We then record, in Section \ref{S:EnumObs}, initial known or straightforward results enumerating braid and commutation classes.

\subsection{Coxeter groups}\label{S:Coxeter}

A \textit{finitely generated Coxeter group} is a group $W$ with a presentation of the form
\[\langle s\in S \mid s^2=1 \text{ for all }s\in S\text{ and }(st)^{m(s,t)}=1 \text{ for all }s,t\in S\rangle,\]
where $S$ is a finite set and $m(s,t)=m(t,s)\in \{2,3,\ldots\}\cup\{+\infty\}$. The elements of $S$ are the \textit{simple reflections}. Because the simple reflections generate the group, any element of $W$ may be written as a product $w=s_{i_1}s_{i_2}\cdots s_{i_k}$, where each $s_{i_j}\in S$. When this $k$ is minimal, it is the \textit{length} of $w$, and this minimal value of $k$ is denoted $\ell(w)$. An expression of $w$ as $s_{i_1}s_{i_2}\cdots s_{i_{\ell(w)}}$ is a \textit{reduced expression} for $w$, and the word formed by the indices $i_1i_2\cdots i_{\ell(w)}$ is a \textit{reduced word} for $w$.  The set of reduced words for $w\in W$ is denoted $R(w)$, and a word formed by taking an ordered subset of the indices in a reduced word $i_1i_2\cdots i_{\ell(w)}$ for $w$, namely $i_{j_1}i_{j_2} \cdots i_{j_m}$ where $1 \leq j_1< \cdots < j_m \leq \ell(w)$, is a \textit{subword}.  A \emph{factor} is a subword in which those indices are consecutive.

In this work, we focus on finite Coxeter groups of type $A$. We should note, however, that the objects in this paper (such as length, reduced expressions, and the commutation and braid classes defined below) have analogues in other types. Indeed some of the results cited below have been proved in that broader context.  The Coxeter group of type $A_{n-1}$ is generated by the simple reflections $S=\{s_1,\ldots,s_{n-1}\}$ subject to the Coxeter relations:
\begin{itemize}
\item[R1.] $s_i^2=1$ for all $i\in[n-1]$,
\item[R2.] $(s_is_{i+1})^3=1$ for all $i\in [n-2]$, and 
\item[R3.] $(s_is_j)^2=1$ for $|i-j|>1$,
\end{itemize}
where $[m]$ denotes the set $\{1, 2, \dots, m\}$. This is, of course, a presentation of the symmetric group $S_{n}$, where the simple reflection $s_i$ is the adjacent transposition swapping $i$ and $i+1$.  Most often, we will express elements of this Coxeter group as reduced words, but occasionally we also use the \textit{window} or \textit{one-line} notation, which records the action of the permutation $w \in S_n$ on the elements of $[n]$ as follows $w = [w(1)\ w(2)\ \cdots \ w(n)]$. A pair $(i, j)$ with $i < j$ is an \textit{inversion} of $w$ if $w(i)>w(j)$, and the length $\ell(w)$ is also equal to the number of inversions in $w$.

Relation R2 is the \textit{braid relation} and can be reformulated as \[s_is_{i+1}s_i=s_{i+1}s_is_{i+1}.\] Relations of the form R3 are \textit{commutation relations}: $s_is_j=s_js_i$ when $|i-j|>1$. We refer to an application of a braid relation as a \textit{braid move}, an application of a commutation relation as a \textit{commutation move}, and an application of either relation as a \textit{Coxeter move}.

\subsection{Braid and commutation classes}\label{S:ClassDefs}

For a fixed permutation $w$, the set $R(w)$ can be partitioned by either of the two Coxeter moves, forming the classes
\begin{align*}
C(w) &:= R(w)/(ij\sim ji) \text{ when } |i-j| > 1, \text{ and}\\
B(w) &:= R(w)/(i(i+1)i \sim (i+1)i(i+1)).
\end{align*}
The first of these, $C(w)$, denotes the \emph{commutation classes} of $R(w)$, and the second, $B(w)$, denotes the \emph{braid classes}. That is, two elements of some $C \in C(w)$ can be obtained from each other by a sequence of commutation moves, whereas elements of some $B \in B(w)$ can be obtained from each other by a sequence of braid moves.

\begin{example}
Consider $[25314] \in S_5$. Then
\begin{align*}
R([25314]) &= \{12432, 14232, 41232, 14323, 41323, 43123\}\\
C([25314]) &= \big\{ \{12432,14232,41232\}, \ \{14323,41323,43123\} \big\}\\
B([25314]) &= \big\{ \{12432\}, \ \{14232, 14323\}, \ \{41232, 41323\}, \ \{43123\} \big\}.
\end{align*}
\end{example}

Consider a reduced word $\mathbf{u}=u_1\cdots u_k$. Let $\b_i$ denote the braid relation that affects the subword $u_{i-1}u_iu_{i+1}$ in the case when $u_{i-1} = u_{i+1}$ and $|u_{i\pm1}-u_i| = 1$, and acts as the identity otherwise. Let $\c_i$ denote the commutation relation that swaps $u_i$ with $u_{i+1}$ if $|u_i - u_{i+1}| > 1$, and acts as the identity otherwise.

\begin{definition}
If $\b_i(\mathbf{u}) \neq \mathbf{u}$ then $\mathbf{u}$ \emph{supports} a braid move in position $i$; equivalently, the subword $u_{i-1}u_iu_{i+1}$ \emph{supports} a braid move. If $\c_i(\mathbf{u}) \neq \mathbf{u}$ then $\mathbf{u}$ \emph{supports} a commutation move in position $i$; equivalently, the subword $u_iu_{i+1}$ \emph{supports} a commutation move. 
\end{definition}

\begin{example}
We have $\b_4(14232) = 14323$ and $\c_2(14232) = 12432$, while $\b_2(14232) = 14232$ because the factor $142$ does not support a braid move, and $\c_4(14232) = 14232$ because the factor $32$ does not support a commutation move.
\end{example}

It is sometimes relevant to know whether a pair of Coxeter moves can act on intersecting factors in a reduced word. Note that it is impossible for a reduced word to support both $\b_i$ and $\b_{i+1}$.

\begin{definition}
If a reduced word $\mathbf{u}$ supports both $\b_i$ and $\b_j$ for $i \neq j$, then $\b_i$ and $\b_j$ act as an \emph{overlapping pair} if $|i-j| = 2$, and act \emph{independently} otherwise. 
If a reduced word $\mathbf{u}$ supports both $\c_i$ and $\c_j$ for $i \neq j$, then $\c_i$ and $\c_j$ act as an \emph{overlapping pair} if $|i-j| = 1$, and act \emph{independently} otherwise. 
\end{definition}

\begin{example}
In the reduced word $1216343 \in R([3254176])$, the braid moves $\b_2$ and $\b_6$ act independently, and the commutation moves $\c_3$ and $\c_4$ act as an overlapping pair.
\end{example}

For any permutation $w$, the braid moves supported by a reduced word $\mathbf{u} \in R(w)$ must either be entirely disjoint, or may overlap in a single letter. Because commutations are not allowed within a braid class $B \in B(w),$ this forces a certain structure on the elements of $B$, which has been studied by Zollinger \cite{zollinger1994equivalence} and also by Bidari and Ernst \cite{Ernst}. 

\subsection{Graphical representations}\label{sec:initialobservations}

The reduced words of a permutation $w$ can be drawn as the vertices of a graph $G(w)$, where two words are connected by a \emph{commutation edge}, or \emph{c-edge}, if they differ by a commutation move, and by a \emph{braid edge}, or \emph{b-edge}, if they differ by a braid move. 

\begin{example}\label{ex:graph for 25314}
In the following graph of reduced words for the permutation $[25314]$, c-edges are indicated by solid lines and b-edges are indicated by dashed lines.
$$
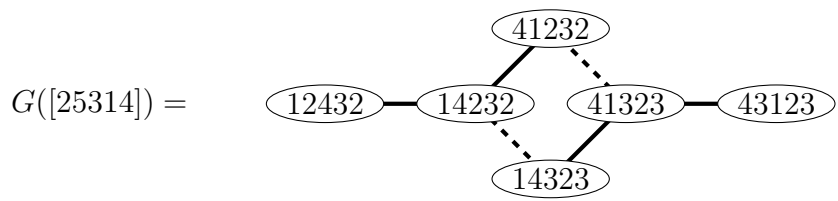
\begin{tikzpicture}
\draw (-3,0) node {$G([25314]) = $};
\draw[ultra thick,dashed] (2.2,-.2) -- (2.8,-.8);
\draw[ultra thick,dashed] (3.8,.2) -- (3.2,.8);
\draw[ultra thick] (.6,0) -- (1.4,0);
\draw[ultra thick] (3.8,-.2) -- (3.2,-.8);
\draw[ultra thick] (2.2,.2) -- (2.8,.8);
\draw[ultra thick] (4.6,0) -- (5.4,0);
\node[fill,white,ellipse,minimum size=.5cm, inner sep=1pt] at  (0,0) {$12432$};
\node[fill,white,ellipse,minimum size=.5cm, inner sep=1pt] at  (2,0) {$14232$};
\node[fill,white,ellipse,minimum size=.5cm, inner sep=1pt] at  (3,-1) {$14323$};
\node[fill,white,ellipse,minimum size=.5cm, inner sep=1pt] at  (3,1) {$41232$};
\node[fill,white,ellipse,minimum size=.5cm, inner sep=1pt] at  (4,0) {$41323$};
\node[fill,white,ellipse,minimum size=.5cm, inner sep=1pt] at  (6,0) {$43123$};
\node[draw,ellipse,minimum size=.5cm, inner sep=1pt] at  (0,0) {$12432$};
\node[draw,ellipse,minimum size=.5cm, inner sep=1pt] at  (2,0) {$14232$};
\node[draw,ellipse,minimum size=.5cm, inner sep=1pt] at  (3,-1) {$14323$};
\node[draw,ellipse,minimum size=.5cm, inner sep=1pt] at  (3,1) {$41232$};
\node[draw,ellipse,minimum size=.5cm, inner sep=1pt] at  (4,0) {$41323$};
\node[draw,ellipse,minimum size=.5cm, inner sep=1pt] at  (6,0) {$43123$};
\end{tikzpicture}$$
\end{example}

The graph $G(w)$ has nice properties, many of which were studied by Reiner and Roichman in \cite{reiner-roichman}.  We will use the fact that this graph is connected in a critical way, and so we record this property separately in the following theorem.

\begin{theorem}[\cite{Matsumoto, Tits}]\label{thm:connected graph}
The graph $G(w)$ is connected.
\end{theorem}

Let $G_c(w)$ be the graph that results from contracting the c-edges of $G(w)$, and let $G_b(w)$ be the graph that results from contracting the b-edges of $G(w)$. The vertices of these two new graphs correspond to the commutation classes $C(w)$ and the braid classes $B(w)$, respectively. Two commutation (respectively, braid) classes are adjacent if an element of one can be obtained from an element of the other by a braid (respectively, commutation) move.

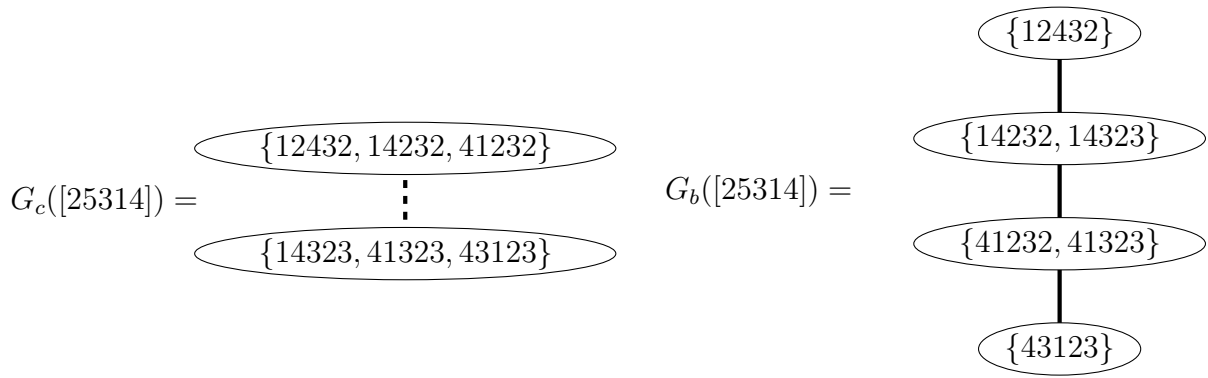
\begin{example}
$$\raisebox{.5in}{\begin{tikzpicture}
\draw (-4,0) node {$G_c([25314]) = $};
\draw[ultra thick,dashed] (0,.7) -- (0,-.7);
\node[fill,white,ellipse,minimum size=.5cm, inner sep=1pt] at (0,.7) {$\{12432,14232,41232\}$};
\node[fill,white,ellipse,minimum size=.5cm, inner sep=1pt] at (0,-.7) {$\{14323,41323,43123\}$};
\node[draw,ellipse,minimum size=.5cm, inner sep=1pt] at  (0,.7) {$\{12432,14232,41232\}$};
\node[draw,ellipse,minimum size=.5cm, inner sep=1pt] at  (0,-.7) {$\{14323,41323,43123\}$};
\end{tikzpicture}}\hspace{.2in}\begin{tikzpicture}
\draw (-4,0) node {$G_b([25314]) = $};
\draw[ultra thick] (0,2.1) -- (0,-2.1);
\node[fill,white,ellipse,minimum size=.5cm, inner sep=1pt] at (0,2.1) {$\{12432\}$};
\node[fill,white,ellipse,minimum size=.5cm, inner sep=1pt] at (0,.7) {$\{14232,14323\}$};
\node[fill,white,ellipse,minimum size=.5cm, inner sep=1pt] at (0,-.7) {$\{41232,41323\}$};
\node[fill,white,ellipse,minimum size=.5cm, inner sep=1pt] at (0,-2.1) {$\{43123\}$};
\node[draw,ellipse,minimum size=.5cm, inner sep=1pt] at (0,2.1) {$\{12432\}$};
\node[draw,ellipse,minimum size=.5cm, inner sep=1pt] at (0,.7) {$\{14232,14323\}$};
\node[draw,ellipse,minimum size=.5cm, inner sep=1pt] at (0,-.7) {$\{41232,41323\}$};
\node[draw,ellipse,minimum size=.5cm, inner sep=1pt] at (0,-2.1) {$\{43123\}$};
\end{tikzpicture}$$
\end{example}

In general, the graphs $G_c(w)$ and $G_b(w)$, along with their analogues in other types, share a certain structural feature.

\begin{theorem}[{\cite[Theorem 3.1]{bergeron2015fan}}]\label{thm:bipartite graphs}
The graphs $G_c(w)$ and $G_b(w)$ are bipartite.
\end{theorem}

\subsection{Enumerative observations}\label{S:EnumObs}

Recall that the goal of this paper is to leverage the fact that braid classes and commutation classes are partitions of the same set, so that we can study them simultaneously in relation to the collection of reduced words of a permutation. The first step in this process is to recognize that braid moves and commutation moves are, in a sense, orthogonal to each other. Proposition \ref{prop:braids and commutations are different} recovers the result \cite[Corollary 10]{zollinger1994equivalence}, proved here in a different manner.

\begin{proposition}\label{prop:braids and commutations are different}
Fix a permutation $w$. For all $B \in B(w)$ and $C \in C(w)$,
$$|B \cap C| \le 1.$$
\end{proposition}

\begin{proof}
Suppose that $\mathbf{u}\neq \mathbf{v}$ are two reduced words for $w$ and that $\mathbf{u}$ and $\mathbf{v}$ are in the same braid class. Since these two words are distinct, 
there is a leftmost letter in which they differ. By assumption, the word $\mathbf{v}$ can be obtained from $\mathbf{u}$ by a sequence of braid moves, which means that this leftmost difference must be $i$ in one of the words and $i+1$ in the other, for some $i$. Therefore, the subword obtained from $\mathbf{u}$ by restricting to the letters ``$i$'' and ``$i+1$'' is not the same as the subword obtained from $\mathbf{v}$ using this same restriction. Since commutation relations cannot change the subword formed by reading only the occurrences of $i$ and $i+1$, the reduced words $\mathbf{u}$ and $\mathbf{v}$ are not in the same commutation class.
\end{proof}

We can classify all permutations that fall into the extreme cases when either $|C(w)| = 1$ or $|B(w)| = 1$, and these classifications will be relevant in the next section. We list these classifications as Propositions \ref{prop:one commutation class} and \ref{prop:one braid class} below.

If $|C(w)|=1$, then by definition of $C(w)$, any reduced expression for $w$ can be obtained from any other by commutation moves.  Stembridge defines such $w$ as \emph{fully commutatitve} elements \cite{stembridge96}. Characterizing those $w$ such that $|C(w)|=1$ can also be done through a pattern avoidance property. In order to state this property, we recall that a permutation $w \in S_n$ is \textit{$321$-avoiding} if there is no triple of indices, $i < j<k$, for which $w(i)>w(j)>w(k)$.

\begin{proposition}[{\cite[Theorem 2.1]{billey-jockusch-stanley}}]\label{prop:one commutation class}
$|C(w)| = 1$ if and only if $w$ is $321$-avoiding.
\end{proposition}

Characterizing those $w$ such that $|B(w)| = 1$ is also straightforward because this strict condition requires that any two reduced words be related by braid moves and that no commutation moves are ever possible.

\begin{proposition}\label{prop:one braid class}
$|B(w)| = 1$ if and only if $w$ has a reduced word of the form $i(i+1)i$, or of the form $i(i+\varepsilon)(i+2\varepsilon)\cdots(i+k\varepsilon)$ for some $\varepsilon \in \{\pm 1\}$ and $k \ge 0$; equivalently, if and only if any two inversions in $w$ share a letter.
\end{proposition}

\begin{proof}
By Proposition~\ref{prop:braids and commutations are different}, reduced words for $w$ cannot support commutation moves, meaning that all letters must be consecutive in every reduced word for $w$. The only possibilities, then, are reduced words of the form
$$i(i+\varepsilon)(i+2\varepsilon)\cdots(i+k\varepsilon)$$
for $\varepsilon \in \{\pm 1\}$ and some $k \ge 0$, or
$$i(i+1)i.$$
This is equivalent to requiring that any two inversions in $w$ share a letter.
\end{proof}

\begin{example}\
\begin{enumerate}\renewcommand{\labelenumi}{(\alph{enumi})}\addtocounter{enumi}{0}
\item The permutation $[241563] \in S_6$ is $321$-avoiding. Thus it is fully commutative and $|C([241563])| = 1$. More precisely,
$$C([241563]) = \big\{ \{13245, 31245, 13425, 13452, 31425, 31452, 34125, 34152, 34512\}\big\}.$$
\item In the permutation $[124563] \in S_6$, any two inversions share at least one letter. Thus $|B([124563])| = 1$. More precisely,
$$B([124563]) = \big\{ \{345\} \big\}.$$
\end{enumerate}
\end{example}

\section{Charting reduced words and bounding sizes}
\label{sec:bounds}

Throughout this section, fix a permutation $w$. The primary tool for our investigation is an organizational scheme for the reduced words in $R(w)$.  This relies on the fact that the braid classes $B(w)$ and the commutation classes $C(w)$ give two different partitions of the collection $R(w)$ of reduced words of a permutation.

\subsection{Bounds on the number of reduced words}

In this subsection, we develop bounds and relationships among the sizes of $R(w)$, $B(w)$, and $C(w)$. These bounds are sharp, and the subsequent subsections will be devoted to characterizing when the upper and lower bounds are each achieved.

\begin{definition}
Fix a permutation $w$, and let $\mathcal{T}(w)$ be a table whose rows are indexed by $B(w)$, and whose columns are indexed by $C(w)$.  For each $\mathbf{u} \in R(w)$, if $\mathbf{u} \in B \in B(w)$ and $\mathbf{u} \in C \in C(w)$, then place $\mathbf{u}$ in row $B$ and column $C$ of the table $\mathcal{T}(w)$.
\end{definition}

Of course, the choice of order for the rows and columns of $\mathcal{T}(w)$ has an impact on the table itself, but it has no effect on the properties we will discuss (and care about) below.

\begin{example}\label{ex:chart for 25314}
Let $w = [25314]$. The braid classes and commutation classes of this element were discussed in previous examples, and we now label them as shown.
\begin{align*}
B_1 &= \{12432\}\\
B_2 &= \{14232,14323\}\\
B_3 &= \{41232,41323\}\\
B_4 &= \{43123\}\\
C_1 &= \{12432,14232,41232\}\\
C_2 &= \{14323,41323,43123\}
\end{align*}
The corresponding table, $\mathcal{T}(w)$, for elements of $R(w)$ is as follows.
$$\begin{array}{c||c|c}
\raisebox{.1in}[.2in][.1in]{}& C_1 & C_2\\
\hline
\hline
\raisebox{.1in}[.2in][.1in]{}B_1 & \ 12432 \ & \\
\hline
\raisebox{.1in}[.2in][.1in]{}B_2 & \ 14232 \ & \ 14323 \ \\
\hline
\raisebox{.1in}[.2in][.1in]{}B_3 & \ 41232 \ & \ 41323 \ \\
\hline
\raisebox{.1in}[.2in][.1in]{}B_4 & & \ 43123 \
\end{array}$$
\end{example}

As is obvious in Example~\ref{ex:chart for 25314}, some cells of $\mathcal{T}(w)$ may be empty. On the other hand, there is a limit to how many reduced words each cell might contain.

\begin{remark}\label{rem:cells have at most one element}
Proposition~\ref{prop:braids and commutations are different} can be restated as follows: each cell in $\mathcal{T}(w)$ contains at most one element.
\end{remark}

Certain features of these charts dictate how many cells must be nonempty. Since these features do not depend on the type of objects filling the cells of the charts (in our case, reduced words), we state and prove the following result in greater generality.

\begin{proposition}\label{prop:nonempty cells in tables with jump property}
Let $T$ be an $r \times c$ rectangular array of data, some of whose cells may be empty. Suppose that:
\begin{itemize}
\item each column of $T$ contains at least one nonempty cell, 
\item each row of $T$ contains at least one nonempty cell, and
\item one can get from any nonempty cell in $T$ to any other nonempty cell by a sequence of in-row or in-column jumps, always from nonempty cells to nonempty cells.
\end{itemize}
Then $T$ contains at least $c+r-1$ nonempty cells.
\end{proposition}

\begin{proof}
Without loss of generality, assume that $c \ge r$. Then, by the first requirement, the array $T$ must have at least $c$ nonempty cells.

If $r > 1$, then the remaining requirements for $T$ imply that there must be a column of $T$ with multiple nonempty cells. Let $x$ denote the number of columns in the array with multiple nonempty cells. In order to be able to reach nonempty cells in all $r$ rows, these $x$ columns must contain at least $r+x-1$ nonempty cells.

Because we initially accounted for one nonempty cell in each column, we now have that $T$ must contain at least
$$c + (r+x-1) - x = c + r - 1$$
nonempty cells.
\end{proof}

\begin{example}

Proposition~\ref{prop:nonempty cells in tables with jump property} is demonstrated in Figure~\ref{fig:table with jump property}, which gives a qualifying $7\times 9$ array with $15$ nonempty cells. 
\begin{figure}[htbp]
$$\begin{tikzpicture}[scale=.75]
\foreach \x in {0,1,2,3,4,5,6,7,8,9} {\foreach \y in {0,1,2,3,4,5,6,7} {\draw (\x,0) -- (\x,7); \draw (0,\y) -- (9,\y);};}
\foreach \x in {(.5,3.5),(.5,5.5),(1.5,.5),(2.5,1.5),(2.5,2.5),(2.5,3.5),(3.5,1.5),(3.5,4.5),(4.5,3.5),(5.5,2.5),(6.5,4.5),(6.5,6.5),(7.5,.5),(7.5,2.5),(8.5,.5)} {\fill \x circle (3pt);}
\end{tikzpicture}$$
\caption{A $7 \times 9$ array satisfying the requirements of Proposition~\ref{prop:nonempty cells in tables with jump property} and containing $15 = 7+9-1$ nonempty cells.}\label{fig:table with jump property}
\end{figure}
Note that a  violation of at least one of the three requirements of Proposition~\ref{prop:nonempty cells in tables with jump property} would occur if any of the nonempty cells in this array were to be made empty. For example, if the cell in the fifth row and eighth column were to be made empty, then there would be no way to get from the bottom right nonempty cell to any nonempty cell outside of its row by a legal sequence of in-row and in-column jumps.
\end{example}

Having organized the elements of $R(w)$ into a chart $\mathcal{T}(w)$ satisfying the requirements of Proposition~\ref{prop:nonempty cells in tables with jump property}, we immediately obtain the following bounds.

\begin{theorem}\label{thm:bounding |R|}
For any permutation $w$,
$$|B(w)| + |C(w)| - 1 \le |R(w)| \le |B(w)| \cdot |C(w)|.$$
\end{theorem}

\begin{proof}
The upper bound in this result is an immediate consequence of the
definition of $\mathcal{T}(w)$ and of Remark~\ref{rem:cells have at
  most one element}. The lower bound follows from
Proposition~\ref{prop:nonempty cells in tables with jump property}. We
note that by the definition of $\mathcal{T}(w)$ and
by Theorem~\ref{thm:connected graph}, the conditions in
Proposition~\ref{prop:nonempty cells in tables with jump property} are
satisfied.
\end{proof}

It is easy to see that these bounds are sharp, because they are simultaneously achieved by any fully commutative $w$. More precisely, if $w$ is fully commutative, then $|C(w)|=1$ and $|B(w)|=|R(w)|$ by Proposition~\ref{prop:braids and commutations are different}, and both bounds of Theorem~\ref{thm:bounding |R|} are achieved. The bounds are similarly realized by any $w$ for which $|B(w)|=1$. 

We devote the next two subsections to characterizing when the upper and lower bounds of Theorem~\ref{thm:bounding |R|} are achieved.

\subsection{Characterizing the upper bound}

We first turn our attention to the upper bound of Theorem~\ref{thm:bounding |R|}. As discussed at the end of the last subsection, fully commutative permutations (for which $|C(w)| = 1$), achieve this bound. The next lemma shows that if either $|C(w)|$ or $|B(w)|$ is equal to $1$, then the other cardinality must be equal to $|R(w)|$.

\begin{corollary}\label{cor:one class of size 1}
Fix a permutation $w$ and suppose that $|C(w)| = 1$ (respectively, $|B(w)| = 1$). Then $|B(w)| = |R(w)|$ (respectively, $|C(w)| =  |R(w)|$).
\end{corollary}

\begin{proof}
This follows directly from Theorem~\ref{thm:bounding |R|}. One could also prove the statement using Proposition~\ref{prop:braids and commutations are different}.
\end{proof}

Therefore, by Corollary~\ref{cor:one class of size 1}, if either $|C(w)|$ or $|B(w)|$ is equal to $1$, then the upper bound of Theorem~\ref{thm:bounding |R|} is automatically achieved. In fact, as we prove below, these are the only permutations that will achieve that upper bound.

\begin{proposition}\label{prop:characterizing the upper bound}
$|R(w)| = |B(w)| \cdot |C(w)|$ if and only if $|B(w)| = 1$ or $|C(w)| = 1$.
\end{proposition}

\begin{proof}
If $|B(w)| = 1$ or $|C(w)| = 1$, then Corollary~\ref{cor:one class of size 1} immediately implies $|R(w)| = |B(w)| \cdot |C(w)|$. This proves one direction of the proposition.

We prove the other direction of the proposition by contrapositive.  Suppose that $|C(w)|>1$ and $|B(w)|>1$. Since $w$ has more than one commutation class, this $w$ is not fully commutative, meaning that there is some reduced word for $w$ that contains a factor of the form $i(i+1)i$. Choose such a reduced word, $\mathbf{u}$, so that the factor $i(i+1)i$ is as far to the left in $\mathbf{u}$ as possible. Since $w$ has more than one braid class, $\mathbf{u}$ must contain at least one additional letter. 
In particular, we have, without loss of generality, that
\begin{equation}\label{eqn:u supporting braid and commutation}
\mathbf{u} = \mathbf{d}ji(i+1)i\mathbf{e},
\end{equation}
with $j>i+1$ and where $\mathbf{d}$ and $\mathbf{e}$ are possibly empty subwords.

Let $B \in B(w)$ and $C \in C(w)$ be the braid and commutation classes, respectively, containing $\mathbf{u} \in R(w)$. From Equation~\eqref{eqn:u supporting braid and commutation}, we see that
$$\mathbf{v}=\b_k\mathbf{u}=\mathbf{d}j(i+1)i(i+1)\mathbf{e}$$
is also an element of $B$, where $k$ is the position of the indicated $i+1$ in $\mathbf{u}$. Additionally,
$$\mathbf{t}=\c_{k-2}\mathbf{u}=\mathbf{d}ij(i+1)i\mathbf{e}$$
is also an element of $C$.

By Proposition~\ref{prop:braids and commutations are different}, in order to satisfy $|R(w)|=|B(w)|\cdot |C(w)|$, the intersection of each braid class with each commutation class must be exactly one reduced word. Let $\mathbf{v} \in C' \in C(w)$ and $\mathbf{t} \in B' \in B(w)$. Thus there must be exactly one reduced word $\mathbf{a} \in B' \cap C'$, in the same commutation class as $\mathbf{v}$ and the same braid class as $\mathbf{t}$. In other words, \[\mathbf{a}=\c_{i_1}\cdots \c_{i_m}\b_k\mathbf{u}=\b_{j_1}\cdots \b_{j_p}\c_{k-2}\mathbf{u},\] for some $m$ and $p$.

As in the proof of Proposition~\ref{prop:braids and commutations are different}, consider the subword of $\mathbf{a}$ obtained by restricting to the entries $i$ and $i+1$. Because $\mathbf{a}$ and $\mathbf{v}$ differ by commutation moves, this subword for $\mathbf{a}$ must be equal to the analogous subword for $\mathbf{v}$. By choice of $\mathbf{u}$, no factor of the form $i(i+1)i$ can be produced in the prefix $\mathbf{d}ij(i+1)$ of $\mathbf{t}$ by Coxeter moves. Therefore any braid moves applied to $\mathbf{t}$ will not change the prefix of the $\{i, i+1\}$ subword formed by using only the letters $i$ and $i+1$, meaning that this subword can never be identical to the corresponding subword for $\mathbf{v}$. As the word $\mathbf{a}$ is obtained from $\mathbf{t}$ using only braid moves, the prefix of the $\{i,i+1\}$ subword of $\mathbf{a}$ cannot be the same as that of $\mathbf{v}$. Therefore there is no such $\mathbf{a}$, and $B' \cap C' = \emptyset$. Hence $|R(w)| \neq |B(w)|\cdot |C(w)|$, concluding the proof.
\end{proof}

The narrowness of Proposition~\ref{prop:characterizing the upper bound} actually implies that those permutations achieving the upper bound of Theorem~\ref{thm:bounding |R|} are also among those that achieve the lower bound of Theorem~\ref{thm:bounding |R|}.

\begin{corollary}\label{cor:upper bound achievers are also lower bound achievers}
If $|B(w)| = 1$ or $|C(w)| = 1$, then
$$|B(w)| + |C(w)| - 1 = |R(w)| = |B(w)| \cdot |C(w)|.$$
\end{corollary}

Note that almost all of the permutations described in Proposition~\ref{prop:one braid class} are fully commutative. In fact, the only elements with a single braid class that are not also fully commutative are those for which $R(w) = \{i(i+1)i,(i+1)i(i+1)\}$ for some $i$. Thus we can reformulate Proposition~\ref{prop:characterizing the upper bound} as follows.

\begin{corollary}\label{cor:characterizing upper bound}
$|R(w)| = |B(w)| \cdot |C(w)|$ if and only if
\begin{itemize}
\item $w$ is fully commutative, or
\item $R(w) = \{i(i+1)i,(i+1)i(i+1)\}$ for some $i$; equivalently, $w(i) = i+2$ and $w(i+2) = i$ for some $i$, and all other values are fixed by $w$.
\end{itemize}
\end{corollary}

The characterization in the previous corollary enables us to enumerate the permutations in $S_n$ achieving the upper bound of Theorem~\ref{thm:bounding |R|}, for any $n$. Note that this is sequence A290953 in \cite{oeis}.

\begin{corollary}\label{cor:upper bound enumeration}
The number of permutations in $S_n$ for which $|R(w)| = |B(w)| \cdot |C(w)|$ is equal to
$$
\begin{cases}
1 & \text{if } n = 1, \text{ and}\\
C_n + n-2 & \text{if } n > 1,
\end{cases}
$$
where $C_n = \binom{2n}{n}/(n+1)$ is the $n$th Catalan number.
\end{corollary}

\begin{proof}
For $n =1$, one of the categories given in Corollary~\ref{cor:characterizing upper bound} is not feasible, and the enumeration is quick to do by hand.

Now assume that $n > 1$. First observe that the two possibilities of Corollary~\ref{cor:characterizing upper bound} are mutually exclusive. Proposition~\ref{prop:one commutation class} says that fully commutative permutations are exactly those that avoid the pattern $321$, and it is well-known that the number of $321$-avoiding permutations in $S_n$ is the $n$th Catalan number $C_n$; see \cite{billey-jockusch-stanley}, for example. It remains to count the permutations $w \in S_n$ for which $R(w) = \{i(i+1)i,(i+1)i(i+1)\}$, and these can be described by choosing the letter $i \in [1,n-2]$.
\end{proof}

\subsection{Characterizing the lower bound}

Characterizing which permutations achieve the lower bound of Theorem~\ref{thm:bounding |R|} has a noticeably different flavor from that of the upper bound characterization. Instead of two classes of permutations, one of which is broad and one of which is quite specific, the lower bound achievers fall into more classes, all of which are highly specified. 

\begin{definition}\label{defn:circuit}
The collection of reduced words $R(w)$ has a \emph{circuit} if there exists a sequence of elements in $R(w)$, say
$$\mathbf{u_0}, \mathbf{u_1}, \mathbf{u_2}, \ldots, \mathbf{u_{2t}} = \mathbf{u_0}$$
for some integer $t$ such that for all $s$, either
$$\begin{cases}
&\mathbf{u_{2s}}, \mathbf{u_{2s+1}} \in B_s \text{ and}\\
&\mathbf{u_{2s+1}}, \mathbf{u_{2s+2}} \in C_s
\end{cases}$$
or
$$\begin{cases}
&\mathbf{u_{2s}}, \mathbf{u_{2s+1}} \in C_s \text{ and}\\
&\mathbf{u_{2s+1}}, \mathbf{u_{2s+2}} \in B_s
\end{cases}$$
where $B_s \in B(w)$ and $C_s \in C(w)$. 
Equivalently, define a graph $\Gamma(w)$ whose vertices are $B(w) \cup C(w)$, with an edge between $B_i$ and $C_j$ if and only if $B_i \cap C_j \neq \emptyset$; then $R(w)$ has a circuit if the graph $\Gamma(w)$ contains a circuit, i.e., if there is some vertex $v$ in $\Gamma(w)$ such that one can travel from $v$ to itself following a nontrivial sequence of edges and vertices. Equivalently, $R(w)$ has a circuit if and only if $\Gamma(w)$ is not a tree.
\end{definition}

\begin{definition}
A permutation $w$ is \emph{circuit-free} if its collection $R(w)$ of reduced words does not have a circuit.
\end{definition}

Avoidance of circuits in the set of reduced words completely characterizes those permutations that achieve the lower bound of Theorem~\ref{thm:bounding |R|}, because having such a circuit would indicate that there are more reduced words than are minimally required.

\begin{proposition}\label{prop: circuit-free means lower bound}
A permutation $w$ is circuit-free if and only if 
\[|R(w)| = |B(w)| + |C(w)| - 1.\]
\end{proposition}

The main idea in the proof of Proposition~\ref{prop: circuit-free means lower bound} is that $R(w)$ has a circuit if and only if removing a reduced word from a chart $\mathcal{T}(w)$ leaves a chart that still satisfies the requirements of Proposition~\ref{prop:nonempty cells in tables with jump property}. In other words, $R(w)$ contains a circuit if and only if one can detour around a given reduced word in a chart $\mathcal{T}(w)$ without violating the rules of Proposition~\ref{prop:nonempty cells in tables with jump property}. 

\begin{proof}[Proof of Proposition~\ref{prop: circuit-free means lower bound}]
Recall that $\Gamma(w)$ is the graph whose vertices are $B(w) \cup C(w)$, with an edge between $B_i$ and $C_j$ if and only if $B_i \cap C_j \neq \emptyset$. By Theorem~\ref{thm:connected graph}, this graph is connected. Recall also that because $R(w)$ has a circuit if and only if the graph $\Gamma(w)$ has a circuit, the permutation $w$ is circuit-free if and only if $\Gamma(w)$ is a tree. 

Recall that a finite connected graph is a tree if and only if it has one fewer edge than vertices.  Due to Proposition~\ref{prop:braids and commutations are different}, each edge in this graph corresponds to exactly one reduced word for $w$, and each reduced word for $w$ corresponds to exactly one edge in this graph. Therefore the number of edges is $|R(w)|$.   The number of vertices in the graph $\Gamma(w)$ is $|B(w)| + |C(w)|$ by definition.

Thus we have that $w$ is circuit-free if and only if $\Gamma(w)$ is a tree, and this graph is a tree if and only if
$$|R(w)| = |B(w)| + |C(w)| -1, $$
which means precisely that $w$ achieves the lower bound of Theorem~\ref{thm:bounding |R|}.
\end{proof}

We can now exploit the idea of circuit-free permutations to explicitly describe the permutations achieving the lower bound of Theorem~\ref{thm:bounding |R|}.

\begin{proposition}\label{prop:characterizing lower bound achievers}
A permutation $w$ is circuit-free (that is, it achieves the lower bound of Theorem~\ref{thm:bounding |R|}) if and only if 
\begin{itemize}\setlength{\itemsep}{.1in}
\item $|B(w)| = 1$, or
\item $|C(w)| = 1$, or
\item $\mathbf{u}i(i+1)i\mathbf{v} \in R(w)$ such that, up to symmetry, one of the following conditions holds:
\begin{itemize}\setlength{\itemsep}{.1in}\renewcommand{\labelitemii}{$\star$}
\vspace{.1in}
\item $\mathbf{u}=\emptyset$ and $\mathbf{v}\in\{(i-1)i,(i-1)(i-2)\cdots(i-1-t)\}$ for some $t \ge 0$, or
\item $\mathbf{u} = \mathbf{v} = i-1$, or
\item $\mathbf{u} = (i-1-t)\cdots (i-2)(i-1)$ and $\mathbf{v} = (i+2)(i+3)\cdots(i+2+t')$ for some $t,t' \ge 0$,
\end{itemize}
where the allowable symmetries are those that arise from the standard operations of reversal, complementation, and reverse complementation.
\end{itemize}
\end{proposition}

\begin{proof}
First note that if $|B(w)|=1$ or $|C(w)|=1$, then $w$ is circuit-free by Corollary~\ref{cor:upper bound achievers are also lower bound achievers} and Proposition~\ref{prop: circuit-free means lower bound}.

Now assume that $w$ is circuit-free, and that $|B(w)| > 1$ and $|C(w)| > 1$. Then $w$ has a reduced word of the form
$$\mathbf{u}i(i+1)i\mathbf{v},$$
where $\mathbf{u}$ and $\mathbf{v}$ could each (but not both) be empty subwords.

If $\mathbf{u}$ supports a commutation move, say $\c _k \mathbf{u}=\mathbf{u'}$, then the four words
$$\mathbf{u}i(i+1)i\mathbf{v}\,, \ \mathbf{u}(i+1)i(i+1)\mathbf{v}\,, \ \mathbf{u'}(i+1)i(i+1)\mathbf{v}\,, \ \mathbf{u'}i(i+1)i\mathbf{v}\,$$
form a circuit in $R(w)$. Thus, since $w$ is circuit-free, then $\mathbf{u}$ (and, analogously, $\mathbf{v}$) cannot support any commutation moves, meaning that adjacent letters are consecutive.

If $\mathbf{u} = \mathbf{\overline{u}}x$, where $x$ is a letter that commutes with both $i$ and $i+1$ (that is, $x \not\in \{i-1,i,i+1,i+2\}$), then this would also produce a circuit, consisting of the four elements
$$\mathbf{\overline{u}}x i(i+1)i \mathbf{v}\,, \
\mathbf{\overline{u}} i(i+1)i x \mathbf{v}\,, \
\mathbf{\overline{u}} (i+1)i(i+1) x \mathbf{v}\,, \
\mathbf{\overline{u}} x (i+1)i(i+1) \mathbf{v}$$
in $R(w)$. Thus if $\mathbf{u}$ is nonempty, then its rightmost letter is either $i-1$ or $i+2$. Similarly, if $\mathbf{v}$ is nonempty, then its leftmost letter is also either $i-1$ or $i+2$.

To summarize these findings so far, if $\mathbf{u}$ (respectively, $\mathbf{v}$) is nonempty, then its letters must be sequentially consecutive with rightmost (resp., leftmost) letter equal to $i-1$ or $i+2$. Call this \textbf{Rule 1}.

If $\mathbf{u}$ is non-monotonic, then it supports a braid move, say $\b_k \mathbf{u} = \mathbf{u'}$. Then $\mathbf{u'} i(i+1)i \mathbf{v} \in R(w)$ and the nonempty word $\mathbf{u'}$ does not satisfy \textbf{Rule 1}: either it is not sequentially consecutive, or its rightmost letter is not equal to $i-1$ or $i+2$. Thus $\mathbf{u}$ and $\mathbf{v}$ must each be monotonic. Call this \textbf{Rule 2}.

If $\mathbf{u} = (i+t)\cdots i(i-1)$, then we claim that $\mathbf{u}$ must contain at most two letters. 
Suppose not; that is, suppose that $t \ge 1$, and let ${\bf \overline{u}} = \b_{t+2}{\bf u}$. Then $R(w)$ would have a circuit because
$$(\c_{t+3}\c_{t})(\b_{t+2}\b_{t+4})(\c_{t+2})(\b_{t+1}\b_{t+3})(\c_{t+4}\c_{t+1})(\b_{t+3}\b_{t+1})(\c_{t+2})(\b_{t+4}\b_{t+2}){\bf \overline{u}} = {\bf \overline{u}}.$$ 
Therefore, by this and symmetric arguments, we conclude that if the second rightmost (respectively, leftmost) letter of $\mathbf{u}$ (resp., $\mathbf{v}$) is $i$ or $i+1$, then $\mathbf{u}$ (resp., $\mathbf{v}$) has no additional letters. Call this \textbf{Rule 3}.

We now use \textbf{Rules 1--3} to show that a reduced word for $w$ must have one of the forms given in the statement of the proposition. For each of those, we check using Proposition~\ref{prop: circuit-free means lower bound} that the resulting permutation $w$ is indeed circuit-free, which will conclude the proof.  To this end, consider the following three cases, which are exhaustive, up to symmetry.
\begin{enumerate}
\item[(i)] $\mathbf{u}$ and $\mathbf{v}$ each have at most one letter. The possible reduced words for $w$ are as follows.
\begin{enumerate}\renewcommand{\labelenumi}{\arabic{enumi}.}
\item ${i(i+1)i}(i-1)$ and symmetric cases\\
$|R(w)| = 3$, $|B(w)| = 2$, $|C(w)| = 2$, and $3 = 2+2-1$.
\item $(i-1){i(i+1)i}(i-1)$ and symmetric cases\\
$|R(w)| = 6$, $|B(w)| = 4$, $|C(w)| = 3$, and $6 = 4+3-1$.
\item $(i+2){i(i+1)i}(i-1)$ and symmetric cases\\
$|R(w)| = 4$, $|B(w)| = 3$, $|C(w)| = 2$, and $4 = 3+2-1$.
\end{enumerate}

\item[(ii)] $\mathbf{u} = \emptyset$ and $\mathbf{v}$ has at least two letters. The possible reduced words for $w$ are as follows.
\begin{enumerate}\renewcommand{\labelenumi}{\arabic{enumi}.}\addtocounter{enumi}{3}
\item ${i(i+1)i}(i-1)i$ and symmetric cases\\
$|R(w)| = 5$, $|B(w)| = 3$, $|C(w)| = 3$, and $5 = 3+3-1$.
\item ${i(i+1)i}(i-1)\cdots(i-1-t)$ and symmetric cases\\
$|R(w)| = t+3$, $|B(w)| = t+2$, $|C(w)| = 2$, and $t+3 = t+2 + 2 - 1$.
\end{enumerate}
The case $\mathbf{v} = \emptyset$ and $\mathbf{u}$ has at least two letters follows by symmetry.

\item[(iii)] $\mathbf{v} \neq \emptyset$ and $\mathbf{u}$ has at least two letters. The possible reduced words for $w$ are as follows.
\begin{enumerate}\renewcommand{\labelenumi}{\arabic{enumi}.}\addtocounter{enumi}{5}
\item $i(i-1){i(i+1)i}x\mathbf{\overline{v}}$, where $x \in \{i-1,i+2\}$, and symmetric cases\\
After a braid move, we have $\ub{(i-1)i(i-1)}(i+1)ix\mathbf{\overline{v}} \in R(w)$.  By \textbf{Rule 1} applied to the bracketed subword, the letters of $\mathbf{v'} = (i+1)ix\mathbf{\overline{v}}$ must be sequentially consecutive, which means that $x = i-1$.  On the other hand, the subword $\mathbf{v'} = (i+1)i(i-1)\mathbf{\overline{v}}$ now violates \textbf{Rule 3}, since $\mathbf{v'}$ has more than two letters.  This contradiction implies that all permutations of this form are \emph{not} circuit-free.

\item $\mathbf{\overline{u}}(i-2)(i-1){i(i+1)i}(i-1)\mathbf{\overline{v}}$ and symmetric cases\\
After a braid move and two commutation moves, we can write $\mathbf{\overline{u}}(i-2)(i+1)\ub{(i-1)i(i-1)}(i+1)\mathbf{\overline{v}} \in R(w)$, breaking \textbf{Rule 1} since the subword $\mathbf{u'} = \mathbf{\overline{u}}(i-2)(i+1)$ does not have sequentially consecutive letters.  Therefore, these permutations are \emph{not} circuit-free.

\item $(i-1-t)\cdots (i-1){i(i+1)i}(i+2)\cdots (i+2+t')$ for $t \ge 1$ and $t' \ge 0$, and symmetric cases\\
$|R(w)| = t + t' +4$, $|B(w)| = t+t'+3$, $|C(w)| = 2$, and $t+t'+4 = t+t'+3 + 2 - 1$.
\end{enumerate}
The case $\mathbf{u}\neq \emptyset$ and $\mathbf{v}$ has at least two letters follows by symmetry.
\end{enumerate}
\end{proof}

Having characterized the permutations that achieve the lower bound of Theorem~\ref{thm:bounding |R|}, we now enumerate this class. That is, we count the permutations characterized by Proposition~\ref{prop:characterizing lower bound achievers}. This will provide an analogue to the enumeration of upper bound achievers given in Corollary~\ref{cor:upper bound enumeration}. Note that this is sequence A290954 in \cite{oeis}. 

\begin{corollary}\label{cor:lower bound enumeration}
The number of permutations in $S_n$ for which the lower bound of Theorem~\ref{thm:bounding |R|} is achieved is 
$$
\begin{cases}
n & \text{if } n \le 2, \text{ and}\\
\displaystyle{C_n + \frac{n^3-3n^2+8n-21}{3}} & \text{if } n > 2,
\end{cases}
$$
where $C_n = \binom{2n}{n}/(n+1)$ is the $n$th Catalan number.
\end{corollary}

\begin{proof}
For $n \in \{1,2\}$, many of the categories given in Proposition~\ref{prop:characterizing lower bound achievers} are not feasible, and the enumeration is quick to do by hand.

Now assume that $n > 2$. The first two options described in bullet points in the characterization given in Proposition~\ref{prop:characterizing lower bound achievers} are exactly those permutations that (also) achieve the upper bound of Theorem~\ref{thm:bounding |R|}, and these were counted in Corollary~\ref{cor:upper bound enumeration} as $C_n + n-2$.

Note that there is no overlap between the permutations $w$ for which $|B(w)| = 1$ or $|C(w)| = 1$ and those permutations satisfying the third bullet point in the statement of Proposition~\ref{prop:characterizing lower bound achievers}. Thus if we can enumerate the permutations satisfying that third bullet point and add $C_n + n-2$, then we will have the desired total. In other words, it remains only to enumerate permutations $w$ with $\mathbf{u}i(i+1)i\mathbf{v} \in R(w)$ such that, up to symmetry, one of the following conditions holds:
\begin{quote}
\begin{enumerate}
\renewcommand{\labelenumi}{Option \arabic{enumi}:}
\setlength{\itemsep}{.1in}
\vspace{.1in}
\item $\mathbf{u}=\emptyset$ and $\mathbf{v}\in\{(i-1)i,(i-1)(i-2)\cdots(i-1-t)\}$ for some $t \ge 0$, or
\item $\mathbf{u} = \mathbf{v} = i-1$, or
\item $\mathbf{u} = (i-1-t)\cdots (i-2)(i-1)$ and $\mathbf{v} = (i+2)(i+3)\cdots(i+2+t')$ for some $t,t' \ge 0$.
\end{enumerate}
\end{quote}
It is straightforward to check that these three options are disjoint (indeed, as designed). Thus the enumeration we want can be obtained by adding the number of permutations in each of the three categories.

\vspace{.1in}

\noindent \textbf{Option 1:} Including all symmetries, the permutations in this category have reduced words of one of the following forms, where the required braid move has been bracketed for clarity.
\begin{itemize}
\item $\ub{i(i+1)i}(i-1)i$, which can be rewritten in the form $j(j-1)\ub{j(j+1)j}$
\item $\ub{i(i+1)i}(i+2)(i+1)$, which can be rewritten in the form $(j+1)(j+2)\ub{j(j+1)j}$
\item $\ub{i(i+1)i}(i-1)(i-2)\cdots (i-1-t)$
\item $\ub{i(i+1)i}(i+2)(i+3)\cdots (i+2+t)$
\item $(i-1-t)\cdots(i-2)(i-1)\ub{i(i+1)i}$
\item $(i+2+t)\cdots(i+3)(i+2)\ub{i(i+1)i}$
\end{itemize}
Other than the equivalences noted in the first two bullet points above, 
none of these words can be transformed into one of the other configurations by a sequence of braid and commutation moves. There are $n-3$ choices of $i$ for each of the first two categories, all of which yield different permutations, contributing $2(n-3)$ permutations to the total. 
Now consider the remaining four categories. For each $i \in [2, n-2]$, we can choose any $t \in [0, i-2]$, which gives 
\[\sum\limits_{i=2}^{n-2}(i-1) = \frac{(n-2)(n-3)}{2}\] choices for $(i,t)$. Each choice yields a distinct permutation, meaning that these four categories contribute $2(n-2)(n-3)$ permutations to the total. Therefore, the number of permutations satisfying Option 1 is
$$2(n-3) + 2(n-2)(n-3) = 2(n-3)(n-1).$$

\vspace{.1in}

\noindent \textbf{Option 2:} Including all symmetries, the permutations in this category have reduced words of the following form, where the required braid move has been bracketed for clarity.
\begin{itemize}
\item $(i-1)\ub{i(i+1)i}(i-1)$, which can be rewritten in the form $(j+2)\ub{j(j+1)j}(j+2)$
\end{itemize}
Since there are $n-3$ choices for $i$ (or $j$), the number of permutations satisfying Option 2 is
$$n-3.$$

\vspace{.1in}

\noindent \textbf{Option 3:} Including all symmetries, the permutations in this category have reduced words of one of the following forms, where the required braid move has been bracketed for clarity.
\begin{itemize}
\item $(i-1-t) \cdots (i-2)(i-1)\ub{i(i+1)i}(i+2)(i+3)\cdots (i+2+t')$
\item $(i+2+t')\cdots (i+3)(i+2)\ub{i(i+1)i}(i-1)(i-2) \cdots (i-1-t)$
\end{itemize}
Neither of these words can be transformed into the other configuration by a sequence of braid and commutation moves, and any choice of $(i,t,t')$ identifies a unique permutation in this category. For each $i \in [2,n-3]$, we need $t \in [0,i-2]$ and $t' \in [0,n-3-i]$. Thus, for a fixed $i$, there are $(i-1)(n-2-i)$ distinct permutations. It follows that, in total, the number of permutations satisfying Option 3 is
\begin{align*}
2\sum_{i=2}^{n-3}&(i-1)(n-2-i) = 2\left[(n-2)\sum_{i=2}^{n-3}(i-1) - \sum_{i=2}^{n-3}i^2 + \sum_{i=2}^{n-3}i\right]\\
&=(n-2)(n-3)(n-4) - \frac{(n-3)(n-2)(2n-5)}{3} + (n-3)(n-2)\\
&=\frac{(n-2)(n-3)(n-4)}{3}.
\end{align*}

Therefore, the number of permutations achieving the lower bound of Theorem~\ref{thm:bounding |R|} is
$$C_n + n-2 + 2(n-3)(n-1) + n-3 + \frac{(n-2)(n-3)(n-4)}{3},$$
which simplifies to the desired result.
\end{proof}

\section{A proposed connection to weak order}\label{sec:conjecture}

Proposition~\ref{prop:characterizing lower bound
achievers} gives a characterization of permutations $w \in S_n$ that
achieve the lower bound of Theorem \ref{thm:bounding |R|} by describing their possible reduced words. 
Based on experimentation in Maple, it also appears that the shape of the interval
$[e,w] = \{ u \in S_n \mid e \leq u \leq w\}$ in weak order provides an alternative characterization.  Roughly
speaking, permutations whose intervals have a narrow shape
achieve the lower bound in Theorem \ref{thm:bounding |R|}, whereas permutations with a wider interval structure have too many reduced words to achieve this lower bound.  We formalize this observation in Conjecture \ref{conj:lower bound} below.

We require several additional definitions in order to state our
conjecture.

\begin{definition}
The {\em weak order} on Coxeter groups also appears naturally in the study of
reduced words. Let $u$ and $v$ be elements of a finitely generated Coxeter group $W$. Then $u\leq v$
in (right) weak order if there is a sequence of simple reflections
$s_{i_1},\ldots,s_{i_k}$ such that $us_{i_1}s_{i_2}\cdots s_{i_k}=v$
and $\ell(us_{i_1}\cdots s_{i_j})+1=\ell(us_{i_1}\cdots s_{i_j+1})$
for all $j=1,\ldots, k-1$. Each reduced
expression for $w$ then corresponds to a maximal chain from $e$ to $w$ in
the weak order on $W$, and vice versa.
\end{definition}

The set of indices appearing in any, equivalently every, reduced word
for $w$ is the \emph{support} of $w$, and we denote the \emph{size} of
this support by $\operatorname{sup}(w)$. The poset $[e,w]$ is ranked,
and we define $r_i$ to be the number of elements at rank $i$, where $i
\in \{0, \dots, \ell(w)\}$. The {\em width} of $w$ is the maximum of
$\{r_0,r_1,\ldots,r_{\ell(w)}\}$, which we denote by
$\operatorname{wid}(w)$. We have added $\operatorname{wid}(w)$ to the database of
combinatorial statistics; see \cite{FindStat}. Figure~\ref{fig:widthAndSupport}
provides several examples that calculate the width and size of the
support of three different permutations.

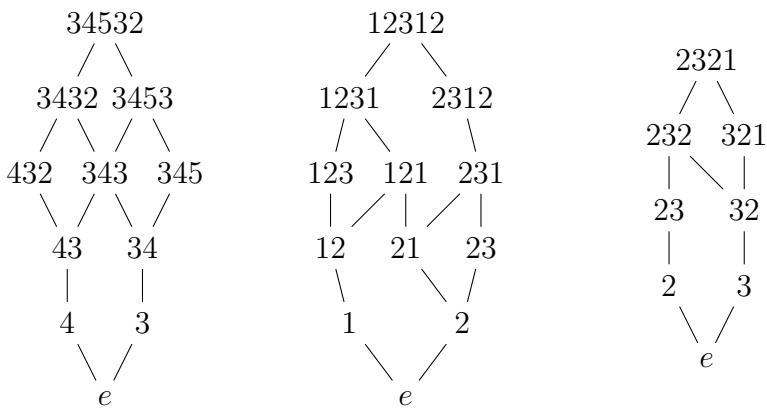
\begin{figure}[ht]
\begin{tikzpicture}
\begin{scope}
\node(34532) at (0,5) {34532};
\node (3432) at (-.5,4){3432};
\node (3453) at (.5,4) {3453};
\node (432) at (-1,3) {432};
\node (343) at (0,3) {343};
\node (345) at (1,3) {345};
\node (43) at (-.5,2) {43};
\node (34) at (.5,2) {34};
\node (4) at (-.5,1) {4};
\node (3) at (.5,1) {3};
\node (e) at (0,0){$e$};
\draw (34532)--(3432);
\draw (34532)--(3453);
\draw (3432)--(432);
\draw (3432)--(343);
\draw (3453)--(343);
\draw (3453)--(345);
\draw (432)--(43);
\draw (343)--(43);
\draw (343)--(34);
\draw (345)--(34);
\draw (43)--(4);
\draw (34)--(3);
\draw (4)--(e);
\draw (3)--(e);
\end{scope}

\begin{scope}[shift={(4,0)}]
\node (12312) at (0,5) {12312};
\node (1231) at (-.75,4) {1231};
\node (2312) at (.75,4) {2312};
\node (123) at (-1,3) {123};
\node (121) at (0,3) {121};
\node (231) at (1,3) {231};
\node (12) at (-1,2) {12};
\node (21) at (0,2) {21};
\node (23) at (1,2) {23};
\node (1) at (-.75,1) {1};
\node (2) at (.75,1) {2};
\node (e) at (0,0) {$e$};

\draw (12312)--(1231);
\draw (12312)--(2312);
\draw (1231)--(123);
\draw (1231)--(121);
\draw (2312)--(231);
\draw (121)--(12);
\draw (121)--(21);
\draw (123)--(12);
\draw (231)--(21);
\draw (231)--(23);
\draw (12)--(1);
\draw (21)--(2);
\draw (23)--(2);
\draw (1)--(e);
\draw (2)--(e);

\end{scope}

\begin{scope}[shift={(8,.5)}]
\node (2321) at (0,4) {2321};
\node (232) at (-.5,3) {232};
\node (321) at (.5,3) {321};
\node (23) at (-.5,2) {23};
\node (32) at (.5,2) {32};
\node (2) at (-.5,1) {2};
\node (3) at (.5,1) {3};
\node (e) at (0,0) {$e$};

\draw (2321)--(232);
\draw (2321)--(321);
\draw (232)--(23);
\draw (232)--(32);
\draw (321)--(32);
\draw (23)--(2);
\draw (32)--(3);
\draw (2)--(e);
\draw (3)--(e);
\end{scope}
\end{tikzpicture}
\caption{
This figure shows the interval $[e,w]$ for three permutations $w$, where the top element represents one choice from $R(w)$.  For the leftmost figure,  $\operatorname{wid}(w) = 3$ and $\operatorname{sup}(w) = 4$.  In the middle figure, $\operatorname{wid}(w) = \operatorname{sup}(w) = 3$.  On the right,  $\operatorname{wid}(w) = 2$ while $\operatorname{sup}(w)=3$.}
\label{fig:widthAndSupport}
\end{figure}

We now provide a conjectural reformulation of our proposition characterizing the permutations that achieve the lower bound in Theorem \ref{thm:bounding |R|}.  In particular, we rephrase the conditions on the reduced words articulated in Proposition \ref{prop:characterizing lower bound achievers} in terms of the width and support of the permutation. The conjecture has been verified using Maple$^{\text{TM}}$ for $S_n$ through $n=6$.

\begin{conjecture}
\label{conj:lower bound}
The permutation $w$ is circuit-free, and thus achieves the lower bound in Theorem \ref{thm:bounding |R|}, if and only if at least one of the following conditions is satisfied:
\begin{enumerate}
\item $|C(w)|=1$
\item $|B(w)|=1$
\item $\operatorname{wid}(w)=2$ \label{conjCondC}
\item $\operatorname{wid}(w) = \operatorname{sup}(w) = 3$ \label{conjCondD}
\end{enumerate}
\end{conjecture}

\begin{example}
Consider again the three permutations referenced in Figure~\ref{fig:widthAndSupport}. The permutation $[152463]$, with reduced word $34532$, has width three and support four, so
Conjecture~\ref{conj:lower bound} predicts that it should not achieve
the lower bound and indeed we have $|R([152463]) = 6$, while
$|B([152463])|+|C([152463])|-1 = 5$. The permutation $[3421]$, with reduced word $12312$, satisfies
Condition~(\ref{conjCondD}) with width and support three, and it achieves
the lower bound with 
$|R([3421])| = |B([3421])|+|C([3421])|-1 = 5$. Condition~(\ref{conjCondC}) holds for the permutation $[4132]$, with reduced word $2321$, and $|R([4132])| = |B([4132])|+|C([4132])|-1 = 3$.
\end{example}

Of course, since the characterization in Proposition \ref{prop:characterizing lower bound achievers} provides an exhaustive and concrete list of possible reduced expressions for permutations that achieve the desired lower bound, it is straightforward to check case by case that each of those permutations satisfies at least one of the width or support criteria in Conjecture \ref{conj:lower bound}.  The more difficult direction seems to be proving that these conditions on the width and support are sufficient in order for the permutation to achieve the lower bound.  More generally, it would thus be interesting to develop a conceptual 
understanding of precisely how the width of the permutation affects the number of reduced words in terms of the number of braid and commutation classes.  Having such an interpretation in terms of intervals in the weak order might also offer a natural pathway for generalizing some of the results in this paper to other Coxeter groups.

\bibliographystyle{alpha}
\bibliography{main}

\end{document}